\documentclass[11pt]{amsart}
\usepackage{amssymb, enumerate, upref, amsrefs, graphicx}
\newtheorem{theorem}{Theorem}[section] 
\newtheorem{lemma}[theorem]{Lemma}
\newtheorem{proposition}[theorem]{Proposition}

\newtheorem{corollary}[theorem]{Corollary}
\theoremstyle{definition}
\newtheorem{definition}[theorem]{Definition}

\newtheorem{question}[theorem]{Question}
\theoremstyle{remark}

\numberwithin{equation}{section}
\newcommand{\CC}{\mathbb{C}}
\newcommand{\DD}{\mathbb{D}}
\newcommand{\NN}{\mathbb{N}}
\newcommand{\RR}{\mathbb{R}}
\newcommand{\TT}{\mathbb{T}}
\newcommand{\ZZ}{\mathbb{Z}}

\DeclareMathOperator{\supp}{supp}
\DeclareMathOperator{\re}{Re}
\DeclareMathOperator{\im}{Im}
\title{Fourier series of circle embeddings}

\author{Leonid V. Kovalev}
\address{215 Carnegie, Mathematics Department, Syracuse University, Syracuse, NY 13244, USA}
\email{lvkovale@syr.edu}
\thanks{L.V.K. was supported by the National Science Foundation grants DMS-1362453 and DMS-1764266.}

\author{Xuerui Yang}
\address{215 Carnegie, Mathematics Department, Syracuse University, Syracuse, NY 13244, USA}
\email{xyang20@syr.edu}
\thanks{X.Y. was supported by an Undergraduate Summer Research Award from Syracuse University.}


\begin{document}
\baselineskip6mm
\maketitle
\begin{abstract}
We study the Fourier series of circle homeomorphisms and circle embeddings, with the emphasis on Blaschke product approximation and the vanishing of Fourier coefficients. The analytic properties of the Fourier series are related to the geometry of the circle embeddings, and have implications for the curvature of minimal surfaces.  
\end{abstract}
\section{Introduction}   

This paper concerns sense-preserving embeddings $f$ of the unit circle $\TT=\{z\in\CC\colon |z|=1\}$ into the complex plane $\CC$, with the emphasis on the relation between the geometry of the image $f(\TT)$ and the behavior of the Fourier coefficients $\hat f\colon \ZZ\to\CC$. The main questions addressed here are: how often does the Fourier series of a circle homeomorphism ($f(\TT)=\TT$) terminate in either positive or negative direction? And does a circle embedding have to have nonzero Fourier coefficients in some fixed finite subset of $\ZZ$? 

Our main results are Theorem~\ref{HSC1}, Theorem~\ref{avoidable}, and Theorem~\ref{horconvex}. Theorem~\ref{HSC1} shows that functions of the form $\arg (B(\zeta)/\zeta^n)$ are dense in $C^1(\TT)$, improving a $C^0$-approximation result due to Helson and Sarason~\cite{HelsonSarason}. As a corollary, circle diffeomorphisms of the form $\zeta\mapsto B(\zeta)/\zeta^n$ are uniformly dense among all sense-preserving circle homeomorphisms (Corollary~\ref{dense}). Theorem~\ref{avoidable}
shows that circle embeddings can have arbitrarily long gaps at the beginning of their Fourier series. In the opposite direction, Theorem~\ref{horconvex} provides a Heinz-type estimate $|\hat f(-1)| + |\hat f(1)|\ge c>0$ (cf.~\cite{Heinz52}) for circle embeddings with a horizontally convex image.

The properties of the Fourier coefficients $\hat f(n)$ of circle embeddings $f$ are of interest for multiple reasons. First, $\hat f(n)$ are also the Taylor coefficients of the harmonic extension of $f$ to the unit disk $\DD$. This implies, for example, that the inequality $|\hat f(1)| \le |\hat f(-1)|$ is an obstruction to the injectivity of this extension, by a theorem of Lewy~\cite{Duren}*{p. 20}. 

When $f(\TT)=\TT$, a lower bound on $|\hat f(1)|^2 + |\hat f(-1)|^2$ yields an upper bound on the Gaussian curvature of a minimal graph over the unit disk. This can be viewed as a quantitative form of the Bernstein theorem on minimal surfaces: every minimal graph over $\RR^2$ is a plane.  This relation motivated the Heinz inequality~\cite{Heinz52} with several subsequent improvements until the sharp form was achieved by Hall~\cite{Hall}. The optimal Gaussian curvature bound remains conjectural~\cite{Duren}*{Conjecture 10.3.2}.

Finally, the fact that $|\hat f(1)| > |\hat f(-1)|$ for every circle homeomorphism, in its quantitative form, is a key to the conformally natural extension of circle homeomorphisms devised by Douady and Earle~\cite{DouadyEarle}.

The paper is organized as follows. Sections \ref{sec:terminate} and ~\ref{sec:approximate} concern the circle diffeomorphisms whose Fourier series terminates in one direction; they are closely related to Blaschke products. In Section~\ref{sec:vanishing} we consider the circle embeddings that lack low-frequency Fourier terms. Section~\ref{sec:prelim} collects the necessary background results.

\section{Preliminaries}\label{sec:prelim}

An \textit{embedding} is a map that is a homeomorphism onto its image. For maps $f\colon \TT\to \CC$ this property is equivalent to being continuous and injective. All circle embeddings considered in this paper are sense-preserving. In the special case $f(\TT)=\TT$ we have a circle homeomorphism.

The Fourier coefficients of an integrable function $f\colon \TT\to\CC$ are given by
\[
\hat f(n) = \frac1{2\pi}\int_{0}^{2\pi} f(e^{i\theta})e^{-in\theta}\,d\theta.
\]
We write $\supp \hat f = \{n\in\ZZ\colon \hat f(n) \ne 0\}$. 
The Lebesgue space $L^2(\TT)$ and the sequence space $\ell^2(\ZZ)$ are equipped with inner products
\[
\left< f, g\right>_{L^2(\TT)} = 
\frac{1}{2\pi}\int_0^{2\pi} f(e^{i\theta})\overline{g(e^{i\theta})}\,d\theta
\]
and 
\[
\left<a, b\right>_{\ell^2(\ZZ)}  = \sum_{n\in\ZZ} a_n \overline{b_n}.
\]
Parseval's theorem asserts that $f\mapsto \hat f$ is an isomorphism, with
\begin{equation}\label{Parseval}
    \left< f, g\right>_{L^2(\TT)} = \left< \hat f, \hat g\right>_{\ell^2(\ZZ)}.
\end{equation}

\begin{corollary}\label{ONB} For every measurable function $f\colon \TT\to \TT$ the set of shifted Fourier coefficients $\{\hat f(\cdot -k) \colon k\in\ZZ\}$ is an orthonormal basis of $\ell^2(\ZZ)$.  
\end{corollary}

\begin{proof} By virtue of~\eqref{Parseval}, the statement is equivalent to $\{e^{i k \theta} f(e^{i\theta}) : k\in\ZZ \}$ being an orthonormal basis of $L^2(\TT)$. The latter follows from the multiplication map $g \mapsto fg$ being a unitary operator on $L^2(\TT)$.  
\end{proof}

A complex-valued function is called harmonic if its real and imaginary parts are harmonic. For any continuous function on $\TT$ the Poisson integral provides a continuous extension $F\colon \overline{\DD}\to\CC$, which is harmonic in $\DD$~\cite{Ahlfors}*{p. 169}. The continuity allow us to relate the Taylor coefficients of $F$ to the Fourier coefficients of $f$, as 
\[
\lim_{r\nearrow 1}\frac{1}{2\pi}\int_{0}^{2\pi} F(re^{i\theta})e^{-in\theta}\,d\theta =  \frac{1}{2\pi}\int_{0}^{2\pi} f(e^{i\theta})e^{-in\theta}\,d\theta = \hat f(n). 
\]
Let us record this as a proposition.
\begin{proposition}\label{continuousBoundary} If $f\colon \TT\to\CC$ is continuous, then the series
\[
F(z) = \sum_{n=0}^\infty \hat f(n)z^n + \sum_{n=1}^\infty \hat f(-n)\bar z^n
\]
defines a harmonic function in $\DD$, for which $f$ provides a continuous boundary extension.
\end{proposition}

If $f\colon \TT\to\CC\setminus \{0\}$ is a continuous function, we denote by $\Delta_\TT \arg f$ the change of the continuous argument of $f(z)$ as $z$ travels around $\TT$ once in the positive direction. Note that $\Delta_\TT \arg f$ is the winding number of $f$ around $0$, multiplied by $2\pi$.

A (finite) \textit{Blaschke product} is a function of the form
\[
B(z) = \sigma\prod_{k=1}^n \frac{z-z_k}{1-\overline{z_k}z}
\]
where $z_1,\dots, z_n\in \DD$ and $\sigma\in\TT$. Note that $\Delta_\TT \arg B = 2\pi n$. 
We refer to $n$ as the degree of $B$. A Blaschke product of degree $1$ is a \textit{M\"obius transformation}. All Blaschke products in this paper are finite. The recent book~\cite{Blaschke} is an excellent reference on such products. We will often use the identity~\cite{Blaschke}*{(3.4.7)}, which states that 
\begin{equation}\label{derivativeB}
    \zeta \frac{B'(\zeta)}{B(\zeta)} = \sum_{k=1}^n \frac{1-|z_k|^2}{|\zeta -z_k|^2} =  \sum_{k=1}^n P_\DD(z_k, \zeta)
    \quad (\zeta\in \TT)
\end{equation}
where $P_\DD$ is the Poisson kernel for $\DD$~\cite{Ransford}*{p. 8},
\[
P_\DD(z, \zeta) = \re\left(\frac{\zeta+z}{\zeta-z}\right)
= \frac{1-|z|^2}{|\zeta-z|^2}.
\]
Note that for $\zeta\in\TT$ the quantity $\zeta B'(\zeta)/B(\zeta)$ is the derivative of $\arg B(\zeta)$  with respect to $\arg \zeta$. This derivative can be used for the following elementary characterization of circle homeomorphisms and diffeomorphisms.

\begin{proposition}\label{derivativeF}
Suppose that $f\colon \TT\to \TT$ is a $C^1$-smooth map. Let $g(e^{i \theta}) = \frac{d \arg f(e^{i\theta})}{d \theta}$ be its derivative. Then 
\begin{enumerate}[(a)]
    \item $f$ is a homeomorphism if and only if $\Delta_\TT \arg f = 2\pi$,  $g\ge 0$ on $ \TT$ and $g$ is not identically zero on any nontrivial subarc of $\TT$;
    \item $f$ is a diffeomorphism if and only if $\Delta_\TT \arg f = 2\pi$ and $g >  0$ on $\TT$.
\end{enumerate}
\end{proposition}
The assumption of $g$ not vanishing on any subarc is assured to hold if $f$ is real-analytic. 

\section{Circle homeomorphisms with a terminating Fourier series}\label{sec:terminate}

In this section $f\colon \TT\to\TT$ is a sense-preserving homeomorphism. Our goal is to identify all such homeomorphisms for which $\supp \hat f$ is bounded from above or from below. It is a simple observation that  $\supp \hat f$ can be bounded from both sides only if it is exactly $\{1\}$.

\begin{proposition}\label{suppbounded} 
If $f$ is a circle homeomorphism with $\supp \hat f$ bounded, then $f$ is a rotation.
\end{proposition}

\begin{proof} Let $a = \min \supp \hat f$ and $b=\max\supp \hat f$. If $a<b$, then 
\[
\left<\hat f(a + \cdot), \hat f(b - \cdot) \right>_{\ell^2(\ZZ)} = \hat f(a) \overline{\hat f(b) } \ne 0
\]
contradicting Corollary~\ref{ONB}. Hence $a=b$, which means $f(z) = cz^a$ for some constant $c$. This map is a homeomorphism only when  $a=1$, i.e.,  $f$ is a rotation. \end{proof}

In contrast to the above, there are rich families of circle homeomorphisms whose Fourier series terminates in one direction only.

\begin{proposition}\label{boundedFourier}
Suppose $f\colon \TT\to\TT$ is a circle homeomorphism. Then 
\begin{enumerate}[(a)]
\item $\supp f$ is bounded below if and only if $f(\zeta) = B(\zeta)/\zeta^n$ for some $n\in \NN\cup \{0\}$, where $B$ is a Blaschke product of degree $n+1$;
\item $\supp f$ is bounded above if and only if $f(\zeta) = \zeta^n/B(\zeta)$ for some $n\in \NN$, where $B$ is a Blaschke product of degree $n-1$.
\end{enumerate}
\end{proposition}

\begin{proof} The ``if'' part is clear in both cases: the Fourier series of $B(\zeta)/\zeta^n$ is supported on $\ZZ\cap [-n, \infty)$, while the Fourier series of $\zeta^n/B(\zeta) = \zeta^n \overline{B(\zeta)}$ is supported on $\ZZ\cap (-\infty, n]$. We proceed to ``only if''.

(a) Let $n = \max(0, -\min\supp\hat f)$ and $g(\zeta) = \zeta^{n}f(\zeta)$. By construction $\hat g(k) = 0$ for all $k<0$. By Proposition~\ref{continuousBoundary} the function $g$ has a holomorphic extension to $\DD$ which is continuous on $\overline{\DD}$. Since $|g|\equiv 1$ on $\TT$, it follows that $g$ is a Blaschke product~\cite{Blaschke}*{Theorem 3.5.2}. Let $d$ be the degree of $g$. It follows that $\Delta_\TT \arg g = 2\pi d$. On the other hand,  $\Delta_\TT \arg f = 2\pi$ because $f$ is a homeomorphism. Hence $\Delta_\TT \arg(\zeta^n) =  2\pi (d-1)$, which means $n=d-1$. Part (a) is proved. 

To prove (b), let $n = \max(0, \max\supp\hat f)$ and $g(\zeta) = \zeta^n \overline{f(\zeta)} = \zeta^{n}/f(\zeta)$. By construction $\hat g(k) = 0$ for all $k<0$. As in part (a) we conclude that $g$ is a Blaschke product of some degree $d$. This time, $\Delta_\TT \arg g = 2\pi d$ together with the relation $g(\zeta) =  \zeta^{n}/f(\zeta)$ yield $d = n-1$, completing the proof. 
\end{proof}

Proposition~\ref{boundedFourier} does not yet establish the above claim about having rich families of circle homeomorphisms, because not every quotient of the form stated in Proposition~\ref{boundedFourier} is a homeomorphism of $\TT$. As a warm-up, let us consider the special case $n=2$ of Proposition~\ref{boundedFourier} (b), that is 
\begin{equation}\label{z^2/B}
f(\zeta) = \sigma  \zeta^2 \frac{1-\overline{z_1} \zeta}{\zeta-z_1}  \quad (z_1\in\DD, \ \sigma\in \TT)   
\end{equation}
Letting $B(\zeta) = (\zeta-z_1)/(1-\overline{z_1}\zeta)$, we obtain from ~\eqref{derivativeB} that
\begin{equation}\label{z^2/Bder}
\frac{\zeta f'(\zeta)}{f(\zeta)} = 2 -  \frac{\zeta B'(\zeta)}{B(\zeta)} = 2 - \frac{1-|z_1|^2}{|\zeta-z_1|^2}     
\quad (\zeta\in\TT)
\end{equation}
The expression ~\eqref{z^2/Bder} is minimized when $|\zeta-z_1| = 1-|z_1|$, and thus its minimum value is 
\[
2  - \frac{1+|z_1|}{1-|z_1|} = \frac{1-3|z_1|}{1-|z_1|}.
\]
By Proposition~\ref{derivativeF}, the function ~\eqref{z^2/B} is a circle homeomorphism if and only if $|z_1|\le 1/3$, and is a diffeomorphism if and only if   $|z_1| < 1/3$. 

To treat both cases of Proposition~\ref{boundedFourier} in a unified way, let us consider the quotients of two Blaschke products. These are precisely the rational functions that map $\TT$ to itself. 

\begin{lemma}\label{Bquotient}
Suppose $B_1$ and $B_2$ are finite Blaschke products, that is
\begin{equation} \label{Bquotient1}
B_1(z) = \sigma_1\prod_{k=1}^n \frac{z-z_k}{1-\overline{z_k}z},
\quad
B_2(z) = \sigma_2\prod_{k=1}^m \frac{z-w_k}{1-\overline{w_k}z}
\end{equation}
where $z_1,\dots, z_n, w_1, \dots, w_m\in \DD$ and $\sigma_1,\sigma_2\in\TT$. The quotient $f= B_1/B_2$ is a circle homeomorphism if and only if $n-m=1$ and 
\begin{equation}\label{Poissoncondition}
    \sum_{k=1}^n P_\DD(z_k, \zeta) \ge \sum_{k=1}^m P_\DD(w_k, \zeta)\quad \text{for all } \zeta\in\TT
\end{equation}
where $P_\DD$ is the Poisson kernel. 
If, additionally, strict inequality holds in~\eqref{Poissoncondition} for all $\zeta\in\TT$, then $f$ is a circle diffeomorphism.
\end{lemma}

\begin{proof} Using~\eqref{derivativeB} we get
\[
\frac{\zeta f'(\zeta)}{f(\zeta)} = \frac{\zeta B_1'(\zeta)}{B_1(\zeta)}
- \frac{\zeta B_2'(\zeta)}{B_2(\zeta)} = \sum_{k=1}^n P_\DD(z_k, \zeta) - \sum_{k=1}^m P_\DD(w_k, \zeta)
\]
and then the conclusion follows from Proposition~\ref{derivativeF}.
\end{proof}

Lemma~\ref{Bquotient} raises the question of verifying the condition~\eqref{Poissoncondition}. This condition can be restated in two equivalent ways. First, it is equivalent to saying that the \textit{balayage} of the signed measure 
\[
\sum_{k=1}^n \delta_{z_k}  -\sum_{k=1}^m \delta_{w_k}
\]
onto the boundary of $\DD$ is a positive measure; see ~\cite{Totik}*{p. 6}. Second, it is equivalent to the inequality 
\begin{equation}\label{harmonicineq}
\sum_{k=1}^n h(z_k)  \ge  \sum_{k=1}^m h(w_k)     
\end{equation}
being true for every positive harmonic function $h$ in $\DD$; this follows by expressing $h$ as a Poisson integral. However, neither of these two interpretations is easier to verify in practice than the original condition~\eqref{Poissoncondition}. If the approach used for the special case ~\eqref{z^2/B} is applied to the analysis of ~\eqref{Poissoncondition}, it leads to the sufficient condition 
\begin{equation}\label{sufficient1}
\sum_{k=1}^n \frac{1-|z_k|}{1+|z_k|} \ge \sum_{k=1}^m \frac{1+|w_k|}{1-|w_k|}.
\end{equation}
However, ~\eqref{sufficient1} is unsatisfactory as it lacks the M\"obius invariance property that is inherent in condition~\eqref{Poissoncondition} (this invariance is particularly clear from the form ~\eqref{harmonicineq}). 

Let $d(z,w)$ denote the pseudo-hyperbolic distance between $z$ and $w$, namely
\[
d (z,q)=\left|\frac{z-w}{1-z\overline{w}}\right| \quad
(z, w \in\DD)
\]
By construction, $d$ is invariant under the M\"obius transformations of $\DD$. 
 
\begin{theorem} \label{pseudo}
If the points $z_0,\dots, z_{n}, w_1,\dots,w_n\in \mathbb D$ satisfy the condition
\begin{equation}\label{pseudo1}
d(z_k,w_k) \le 
\frac{(1-d(z_k, z_{0}))(1-d(w_k, z_0))} {4n},
\quad k=1,\dots, n,
\end{equation}
then the quotient 
\[
f(\zeta) = \prod_{k=0}^{n} \frac{\zeta-z_k}{1-\overline{z_k}\zeta} 
\left( \prod_{k=1}^{n} \frac{\zeta-w_k}{1-\overline{w_k}\zeta} \right)^{-1}
\]
is a circle homeomorphism. Furthermore, if~\eqref{pseudo1} is strict for some $k$, then $f$ is a diffeomorphism of $\TT$. 
\end{theorem}

The proof requires the following ``additive Harnack inequality'' for positive harmonic functions in $\DD$. 

\begin{lemma}\label{additive}
For every positive harmonic function $h$ on $\DD$ and all $z, w\in \DD$ we have
\begin{equation}\label{additive1}
 |h(z)-h(w)| \le \frac{2 |z-w|}{(1-|z|)(1-|w|)} h(0).
\end{equation}
\end{lemma}

\begin{proof} 
Since $h$ can be written as a Poisson integral
\[
h(z) = \int_\TT P_\DD(z, \zeta)\,d\mu(\zeta)
\]
for some measure $\mu$ on $\TT$~\cite{DurenHp}*{Theorem 1.1}, it suffices to prove~\eqref{additive1} for $h(z) = P(z, \zeta)$ with $\zeta\in\TT$. For such $h$ we have
\[
\begin{split}
    |h(z)-h(w)|
    & = \left|\re\left(\frac{\zeta + z}{\zeta - z} - \frac{\zeta + w}{\zeta - w} \right)\right| 
    \\ & \le \left|\frac{2(z-w)\zeta }{(\zeta - z)(\zeta-w)} \right|
      \le  \frac{2|z-w|}{(1-|z|)(1-|w|)}
\end{split}
\]
which proves~\eqref{additive1}, as $h(0)=1$.
\end{proof}

Equality is sometimes attained in~\eqref{additive1}, for example if $h(z) = P(z, 1)$, $w=0$, and $0<z<1$. However,~\eqref{additive1} is not sharp for general $z, w\in \DD$. It may be of interest to obtain a sharp form of the additive Harnack inequality. 

\begin{proof}[Proof of Theorem~\ref{pseudo}]  By Proposition~\ref{Bquotient}, we only need to estimate 
\begin{equation}\label{pseudo2}
\sum_{k=0}^{n} P_\DD(z_k, \zeta) - \sum_{k=1}^n P_\DD(w_k, \zeta)
\end{equation}
from below. By composing $f$ with a M\"obius transformation, we can achieve $z_{0}=0$ without affecting the homeomorphism property of $f$. Then~\eqref{pseudo1} takes the form
\begin{equation}\label{pseudo3}
d(z_k,w_k) \le \frac{(1-|z_k|)(1-|w_k|)}{4n},
\quad k=1,\dots, n.
\end{equation}
Using Lemma~\ref{additive} and noting that $|z_k-w_k|\le 2d(z_k, w_k)$, we arrive at 
\[
|P_\DD(z_k, \zeta) - P_\DD(w_k, \zeta)| 
\le \frac{4d(z_k, w_k)}{(1-|z_k|)(1-|w_k|)} \le \frac1n.
\]
Since $P_\DD(0, \zeta) = 1$ on $\TT$, we conclude that~\eqref{pseudo2} is bounded from below by $1 - n(1/n) = 0$.  Furthermore, if $0$ is attained, then equality must hold in~\eqref{pseudo3} and consequently in~\eqref{pseudo1} for all $k$.  
\end{proof}

We can finally show that the circle homeomorphisms of the two types identified by Proposition~\ref{boundedFourier}  indeed exist for all values of $n$. The following corollary of Theorem~\ref{pseudo} is obtained by specializing the theorem to the cases $w_k\equiv 0$ and $z_k\equiv 0$. 

\begin{corollary} Suppose $z_1,\dots, z_n\in \DD$. 
\begin{enumerate}[(a)]
    \item If for $k=1, \dots, n-1$ 
    \[|z_k| \le \frac{1-|z_n|}{4(n-1)}(1-d(z_k, z_n))\]
then the function
\[
f(\zeta) = \zeta^{1-n} \prod_{k=1}^n \frac{\zeta-z_k}{1-\overline{z_k}\zeta}
\]
is a circle homeomorphism. 
\item If $|z_k| \le 1/(4n+1)$ for $k=1,\dots, n$, then the function
\[
f(\zeta) = \zeta^{n+1} \prod_{k=1}^n \frac{1-\overline{z_k}\zeta}{\zeta-z_k}
\]
is a circle homeomorphism. 
\end{enumerate}
\end{corollary}

\section{Approximation by rational circle homeomorphisms}\label{sec:approximate}

The goal of this section is to show that the circle homeomorphisms with a terminating Fourier series (i.e., those identified by the Proposition~\ref{boundedFourier}) are dense in the set of all circle homeomorphisms, with respect to the uniform norm. A similar approximation result for continuous functions was obtained by Helson and Sarason~\cite{HelsonSarason}, see also \cite{Blaschke}*{Theorem 4.3.1}. 

\begin{theorem}\label{HS}~\cite{HelsonSarason}*{page 9} Every continuous function $u\colon \TT\to\RR$ can be approximated uniformly by functions of the form $\arg (B(\zeta)/\zeta^n)$ where $B$ is a Blaschke product and $n$ is the degree of $B$.
\end{theorem}

Given a circle homeomorphism $f\colon \TT\to \TT$, we can take a continuous branch of $u(\zeta) = \arg(f(\zeta)/\zeta)$ on $\TT$ and 
apply Lemma~\ref{HS} to it. It follows that $f$ can be uniformly approximated by functions of the form $B(\zeta)/\zeta^{n-1}$ with $n=\deg B$. The same argument, applied to $v(\zeta) = \arg(\zeta/f(\zeta))$, yields an approximation to $f$ of the form $\zeta^{n+1}/B(\zeta)$.  However, this does not achieve our goal stated above, since it is not guaranteed that the approximating function is a circle homeomorphism. 

We need a stronger form of the Helson-Sarason theorem, with approximation in the $C^1$ norm instead of the uniform norm. A smooth function   $u\colon \TT\to \RR$ can be interpreted as a smooth $2\pi$-periodic function on $\RR$, and we use this interpretation to define its derivative $u'\colon \TT\to\RR$ and the $C^1$ norm $\|u\|_{C^1} = \sup_\TT |u| + \sup_\TT |u'|$. 

\begin{theorem}\label{HSC1} Every $C^1$-smooth function $u\colon \TT\to\RR$ can be approximated in the $C^1$ norm by functions of the form $\arg (B(\zeta)/\zeta^n)$ where $B$ is a Blaschke product and $n$ is the degree of $B$.
\end{theorem}

Before proving Theorem~\ref{HSC1}, let us observe that it implies that circle diffeomorphisms with a terminating Fourier series are uniformly dense in the set of circle homeomorphisms.

\begin{corollary}\label{dense} 
Every circle homeomorphism $f\colon \TT\to \TT$ is a uniform limit of circle diffeomorphisms of the form $\zeta\mapsto B(\zeta)/\zeta^{n-1}$ where $B$ is a Blaschke product of degree $n$. It is also a uniform limit of circle diffeomorphisms of the form $\zeta\mapsto \zeta^{n+1}/B(\zeta)$ where   $B$ is a Blaschke product of degree $n$.  
\end{corollary}

\begin{proof} It is straightforward to approximate $f$ by a diffeomorphisms $g\colon \TT\to \TT$ in the uniform norm. Indeed, $f$ lifts to a continuous strictly increasing function $F\colon \RR\to\RR$ such that $f(e^{i\theta}) = e^{iF(\theta)}$. It is easy to see that the convolution of $F$ with a smooth bump function has strictly positive derivative, and thus descends to a circle diffeomorphism.  

The function $u(\zeta) = \arg(g(\zeta)/\zeta)$ is $C^1$-smooth, with $u'>-1$ on $\TT$. Theorem~\ref{HSC1} provides $C^1$ approximation to $u$ of the form $v(\zeta) = \arg(B(\zeta)/\zeta^n)$. When $\|u-v\|_{C^1}$ is small enough, we have $v'>-1$ and therefore the map $\zeta \mapsto B(\zeta)/\zeta^{n-1}$ is a circle diffeomorphism. 

To prove the second statement, apply Theorem~\ref{HSC1} to the function $u(\zeta) = \arg(\zeta/g(\zeta))$ for which $u'<1$, and conclude as above. 
\end{proof}

Our first step toward the proof of Theorem~\ref{HSC1} is to approximate a continuous function with zero mean by a linear combination of Poisson kernels with integer coefficients. 

\begin{lemma} \label{integer Poisson}
Suppose $h\colon \TT\to\RR$ is continuous and $\int_\TT h = 0$. Then for any $\epsilon>0$ there exist  $n\in \mathbb N$ and $z_k, w_k\in \mathbb D$ ($k=1, \dots, n$) such that 
\begin{equation} \label{intPo1}
\left|h(\zeta) - \sum_{k=1}^n P_\DD (z_k, \zeta) + \sum_{k=1}^n  P_\DD (w_k, \zeta) \right| < \epsilon
\quad \text{for all } \zeta\in\TT.
\end{equation}
\end{lemma}

\begin{proof} Let $H$ be the harmonic extension of $h$ to the unit disk. We fix $r<1$ such that 
\begin{equation}\label{intPo1.5}
|H(r\zeta) - h(\zeta)| < \epsilon/3 \quad \text{for all } \zeta\in\TT. 
\end{equation}
In what follows it will be convenient to consider $h$ as a $2\pi$-periodic function on $\RR$, writing $h(\theta)$ instead of $h(e^{i\theta})$. Similarly, we write $P_r(\theta)$ for $P_\DD(re^{i\theta}, 1)$. 

Since $h$ has zero mean, there exists a $C^1$-smooth $2\pi$-periodic function $g$ such that $g'=h$. The Poisson integral of $h$ can be written in terms of $g$ via the integration by parts:
\begin{equation}\label{intPo2}
\begin{split}
H(re^{i\theta}) & = \frac{1}{2\pi} \int_0^{2\pi} h(\phi) P_r(\phi-\theta)\, d\phi  
\\ & = -\frac{1}{2\pi} \int_0^{2\pi} g(\phi)\frac{\partial P_r(\phi-\theta)}{\partial \phi}\, d\phi
\end{split}
\end{equation}
Denote 
\[
Q_r(\theta)=\frac{\partial P_r(\theta)}{\partial \theta} 
\quad \text{and} \quad 
R_r(\theta)=\frac{\partial^2 P_r(\theta)}{\partial \theta^2}.
\] 
For sufficiently large $n$, the last integral in~\eqref{intPo2} is well approximated by an $n$-point Riemann sum, meaning that
\begin{equation}\label{intPo3}
\left|H(re^{i\theta}) + \frac{1}{2\pi n}\sum_{k=1}^n g(\phi_k)
Q_r(\phi_k-\theta) \right|<\frac{\epsilon}{3}
\end{equation}
where $ \phi_k = \frac{2\pi k}{n}$.

For $1\le k\le n$, let $a_k=-g(\phi_k)/n$. By Taylor's theorem,
\begin{equation}\label{intPo4}
P_r(\phi_k-\theta+a_k)-P_r(\phi_k-\theta) 
- a_kQ_r(\phi_k-\theta)
= \frac{R_r(\psi_k)}{2}a_k^2
\end{equation}
for some $\psi_k$. Both $g$ and $R_r$ are bounded on $\RR$, being continuous and periodic. Therefore,~\eqref{intPo4} implies 
\begin{equation}\label{intPo5}
\left|  P_r(\phi_k-\theta+a_k)-P_r(\phi_k-\theta)-a_k Q_r(\phi_k-\theta) \right|  \le  \frac{C}{n^2}
\end{equation}
with a constant $C$ independent of $n$. Summing over $k$ and letting $n$ be sufficiently large, we conclude that 
\begin{equation}\label{intPo6}
\left|\sum_{k=1}^n \{ P_r(\phi_k-\theta+a_k)-P_r(\phi_k-\theta)-a_k Q_r(\phi_k-\theta)\} \right| < \frac{\epsilon}{3}
\end{equation}
Finally, combine~\eqref{intPo1.5}, ~\eqref{intPo3} and~\eqref{intPo6} to obtain  
\[
\left| h(\theta)-\sum_{k=1}^n \{P_r(\phi_k-\theta+a_k) - P_r(\phi_k-\theta)\} \right| <\epsilon
\]
from where the lemma follows by letting $z_k=re^{i(\phi_k+a_k)}$ and $w_k=re^{i\phi_k}$. 
\end{proof}

The approximation provided by Lemma~\ref{integer Poisson} yields approximation in the $C^1$ norm by a quotient of Blaschke products, as shown below. 

\begin{lemma}\label{quotient approx} Suppose $u\colon \TT\to \RR$ is a $C^1$ smooth function, and the numbers $z_1, \dots, z_n$, $w_1, \dots, w_n  \in \DD$ are such that
\begin{equation} \label{quo1}
\left|u'(\zeta) - 
\sum_{k=1}^n P_\DD (z_k, \zeta)
+ \sum_{k=1}^n P_\DD (w_k, \zeta)
\right| < \epsilon
\end{equation}
for all $\zeta\in\TT$. Then there exists $\sigma\in\TT$ such that the rational function 
\[
B(\zeta) = \sigma \prod_{k=1}^n \frac{\zeta - z_k}{1-\overline{z_k}\zeta}  \frac{1-\overline{w_k}\zeta}{\zeta-w_k}
\]
satisfies $\|u  - \arg B\|_{C^1} < (\pi+1) \epsilon$ for some continuous branch of $\arg B$.
\end{lemma}

\begin{proof} Recall from~\eqref{derivativeB} that the derivative of $\arg B(e^{i\theta})$ with respect to $\theta$ is equal to 
\[
\sum_{k=1}^n P_\DD (z_k, e^{i\theta})
- \sum_{k=1}^n P_\DD (w_k, e^{i\theta})
\]
By~\eqref{quo1}, the derivative of the difference $u(e^{i\theta}) - \arg B(e^{i\theta})$ is less than $\epsilon$ in absolute value. By choosing $\sigma$ so that $\arg B(1) = u(1)$ and using the mean value theorem, we conclude that $|u(e^{i\theta}) - \arg B(e^{i\theta})| < \pi \epsilon$ for all $\theta \in [-\pi, \pi]$. The estimate for the $C^1$ norm of $u-\arg B$ follows.
\end{proof}

The combination of Lemmas~\ref{integer Poisson} and \ref{quotient approx} yields an approximation result similar to Theorem~\ref{HSC1} but with the quotient of two Blaschke products instead of a Blaschke product divided by a monomial. The following result will allow us to shift the poles of the quotient to $0$. 

\begin{lemma} \label{appro a single kernel}
For every $z_0\in \DD$ and every $\epsilon>0$ there exist $n\in \mathbb N$ and $z_1, \dots, z_{n-1}\in \mathbb D$ such that  
\begin{equation}\label{single1}
    \left| \sum_{k=0}^{n-1} P(z_k,\zeta) - n \right| < \epsilon 
\quad \text{for all }\zeta\in \TT. 
\end{equation}
Specifically, we can take $z_k = z_0 \exp(2\pi i k/n)$. 
\end{lemma}

\begin{proof} Fix $\zeta\in \TT$. Let $r=|z_0|$ and choose $R$ such that $r<R<1$. The rational function $\psi (z) = (\zeta+z)/(\zeta-z)$ is holomorphic and bounded by $M=(1+R)/(1-R)$ on the disk $|z|\le R$. Theorem~2.1~\cite{Trapezoid} asserts that uniform Riemann sums converge exponentially fast to the integral of $\psi$ over $|z|=r$, namely
\begin{equation}\label{single2}
\left| \frac{2\pi}{n}\sum_{k=0}^{n-1} \psi(z_k) - \int_0^{2\pi} \psi(re^{i\theta})\, d\theta \right| \le \frac{2\pi M}{(R/r)^n-1}
\end{equation}
where $z_k = z_0 \exp(2\pi i k/n)$. By the mean value property $\int_0^{2\pi} \psi(re^{i\theta})\, d\theta =2\pi \psi(0) = 2\pi $. Multiplying~\eqref{single2} by $n/(2\pi)$  we obtain 
\begin{equation}\label{single3}
\left|\sum_{k=0}^{n-1} \psi(z_k) - n \right| \le \frac{2\pi M n}{(R/r)^n-1}
\end{equation}
The right hand side of~\eqref{single3} tends to $0$ as $n\to \infty$. This and the identity $\re \psi(z_k) = P_\DD(z_k, \zeta)$ yield~\eqref{single1}.  
\end{proof}

The following is a corollary of Lemmas~\ref{integer Poisson} and
~\ref{appro a single kernel}.

\begin{corollary}\label{better Poisson}
Suppose $h\colon \TT\to\RR$ is continuous and $\int_\TT h = 0$. Then for any $\epsilon>0$ there exist  $n\in \mathbb N$ and $z_1,\dots, z_n\in \mathbb D$ such that 
\begin{equation}\label{better0}
\left|h(\zeta) - \sum_{k=1}^n P_\DD (z_k, \zeta) + n \right| < \epsilon
\quad \text{for all } \zeta\in\TT.
\end{equation}
\end{corollary}

\begin{proof} Lemma~\ref{integer Poisson} yields an approximation of the form
\begin{equation}\label{better1}
\max_{\zeta\in\TT}\left|h(\zeta) - \sum_{k=1}^n P_\DD (z_k, \zeta) + \sum_{k=1}^n  P_\DD (w_k, \zeta) \right| < \frac{\epsilon}{2}.
\end{equation}
Then we use Lemma~\ref{appro a single kernel} to replace each term $P_\DD (w_k, \zeta)$ in ~\eqref{better1} by a sum of the form $n_k - \sum_{j=1}^{n_k-1} P(w_{kj}, \zeta)$ such that 
\[
\max_{\zeta \in \TT} \left|P_\DD (w_k, \zeta) - \left\{n_k - \sum_{j=1}^{n_k-1} P(w_{kj}, \zeta) \right\}  \right| < \frac{\epsilon}{2n}
\]
thus arriving at~\eqref{better0} with some larger value of $n$. 
\end{proof}

\begin{proof}[Proof of Theorem~\ref{HSC1}] First apply Corollary~\ref{better Poisson} to the function $h=u'$. Then use Lemma~\ref{quotient approx} with $w_1=\cdots=w_n=0$.
\end{proof}

\section{Vanishing Fourier coefficients}\label{sec:vanishing}

In the investigation of the Fourier coefficients of circle homeomorphisms a special role is played by the first coefficient $\hat f(1)$. Indeed, the identity map $f(\zeta)=\zeta$ has all coefficients other than $\hat f(1)$ equal to zero. In contrast, Hall~\cite{Hall} proved that $\hat f(1)$ never vanishes for circle homeomorphisms. This result cannot be strengthened to a lower bound for $|\hat f(1)|$, as is shown by the M\"obius transformation $f(\zeta) = (\zeta+a)/(1+\bar a\zeta)$ which has $\hat f(0)=a$ and consequently $|\hat f(1)|^2 \le 1-|a|^2$ by Parseval's theorem. However, having large $|\hat f(0)|$ is the only obstruction here: 
Hall~\cite{Hall}*{Theorem 2} gave a positive lower bound for $|\hat f(0)|+|\hat f(1)|$ among all circle homeomorphisms, which 
Weitsman~\cite{Weitsman} sharpened to $|\hat f(0)|+|\hat f(1)| > 2/\pi$, using~\cite{Hall85}.

The most notable estimate for the Fourier coefficients of circle homeomorphisms is $|\hat f(-1)|^2 + |\hat f(1)|^2 \ge 27/(4\pi^2)$, which is a sharp bound obtained by Hall~\cite{Hall} after 30 years of gradual improvements, starting with the paper~\cite{Heinz52} by Heinz. This line of investigation, motivated by curvature estimates for minimal surfaces, remains unfinished: see ~\cite{Duren}*{\S10.3} and ~\cite{Hall98}.   

There are also nonvanishing results for more general circle embeddings.  When $f(\TT)$ is convex, 
the Rad\'o-Kneser-Choquet theorem~\cite{Duren}*{p. 29} states that the harmonic extension of $f$ is a diffeomorphism, and therefore $\hat f(1)\ne 0$; more precisely, $|\hat f(1)| > |\hat f(-1)|$. When $f(\TT)$ is star-shaped about $0$, 
Hall~\cite{Hall}*{Theorem 2} proved that $|\hat f(0)|+ |\hat f(1)| > 0 $. The following proposition shows that the term $|\hat f(0)|$ is necessary here.

\begin{proposition} \label{star}
 There exists an embedding $f\colon \TT\to \CC$ such that $f(\TT)$ is star-shaped about $0$ and $\hat f(1)=0$.
\end{proposition}

\begin{proof} We look for $f$ such that  $\overline{f(\zeta)}=f(\bar{\zeta})$ for all $\zeta\in\TT$, which ensures that $f(\TT)$ is symmetric about the real axis, and that $\hat f$ is real-valued. Let us write $f(e^{i\theta}) = F(\theta)$ for $\theta\in [0,\pi]$, then 
\[ \hat f(n) = \frac{1}{\pi}\re \int_0^\pi F(\theta)e^{-in\theta} d\theta. \]
Our $F$ will be piecewise linear, which justifies integration by parts:
\[
\pi \hat f(1) =\re  i F(\theta)e^{-i\theta}\bigg|_{0}^{\pi} 
- \re \int_0^\pi i F'(\theta) e^{-i\theta} d\theta = \im \int_0^\pi F'(\theta) e^{-i\theta} d\theta
\]
where the boundary term has zero contribution because $F(0)$ and $F(\pi)$ are real. 

We choose $F$ piecewise linear with $F(0)=1$, $F(\pi)=-1$, and $F(2\pi /3) = x+iy$ where $x, y>0$ are to be chosen later. Thus, 
\[
F'(\theta) = \begin{cases} \frac{3}{2\pi} (x-1+iy),\quad & 0<\theta<2\pi/3; \\ 
\frac{3}{\pi}(-1-x-iy),\quad & 2\pi/3 < \theta < \pi 
\end{cases}
\]
which yields
\[
\begin{split}
\im \int_0^\pi F'(\theta) e^{-i\theta} d\theta & = - \frac{3(x-1)}{2\pi}  \int_0^{2\pi/3} \sin \theta\,d\theta 
+ \frac{3y}{2\pi}  \int_0^{2\pi/3} \cos \theta\,d\theta \\ 
&+ \frac{3(x+1)}{\pi} \int_{2\pi/3}^\pi  \sin \theta\,d\theta
-  \frac{3y}{\pi} \int_{2\pi/3}^\pi  \cos \theta\,d\theta \\ 
& = - \frac{9(x-1)}{4\pi}   + \frac{3\sqrt{3} y}{4\pi} + \frac{3(x+1)}{2\pi} +  \frac{3\sqrt{3}y}{2\pi} \\
& = \frac{3}{4\pi} \left(-x+3\sqrt{3}y +5\right) 
\end{split}
\]
For example, we achieve $\hat f(1)=0$ with the choice $(x, y)=(8, 1/\sqrt{3})$.  The curve $f(\TT)$ is a non-convex quadrilateral with vertices $1, x+iy, -1, x-iy$, which is obviously star-shaped about 0. 
\end{proof}

By subtracting a constant from $f$ in Proposition~\ref{star} we can achieve  $\hat f(0) = \hat f(1) = 0$; of course, $f(\TT)$ will no longer be star-shaped about $0$ then. On the other hand, for every circle embedding $\hat f(n)$ must be nonzero for some positive $n$. Indeed, a computation with Green's formula shows that the area enclosed by $f(\TT)$ is $\pi \sum_{n\in\ZZ} n|\hat f(n)|^2$, which implies $\hat f(n)\ne 0$ for some $n>0$. This raises the question: is there a fixed integer $N$ such that 
\begin{equation}\label{nonvanishing}
    \sum_{n=-N}^N |\hat f(n)| > 0 
\end{equation}
for every circle embedding $f$? By Hall's theorem, $N=1$ suffices when $f(\TT)$ is star-shaped about $0$. The main result of this section shows there is no universal $N$ for general circle embeddings.

\begin{theorem}\label{avoidable} For every $N\in\NN$ there exists an embedding $f\colon \TT\to\CC$ such that $\hat f(n)=0$ whenever $|n|\le N$. 
\end{theorem}

\begin{proof} Let $k=N+2$. We will consider embeddings $f\colon \TT\to \CC$ with $k$-fold symmetry, that is 
\begin{equation}\label{kfold}
    f(e^{2\pi i/k}\zeta) = e^{2\pi i/k}f(\zeta),\quad \zeta\in\TT. 
\end{equation}
As a consequence of ~\eqref{kfold}, the Fourier coefficients of $f$ satisfy
\[
e^{ 2\pi n i / k} \hat f(n)  = e^{2\pi i/k}  \hat f(n), \quad n\in\ZZ, 
\]
which implies $\hat f(n) = 0$ for all $n$ such that $n\not\equiv 1 \bmod  k$. It remains to construct an embedding $f$ such that~\eqref{kfold} holds and $\hat f(1)=0$, which will assure  $\hat f(n)=0$ for $1-k < n < 1+k$. 

For $\theta\in\RR$ we define
\[
g(\theta) = \arccos(\cos \theta), \quad 
\rho(\theta) = 1 + \frac{2}{\pi} g(\theta),\quad 
h(\theta) = g(\theta) + \frac{1}{\pi} g(\theta)^2.
\]
Since $\rho$ and $h$ are $2\pi$-periodic and continuous, the function 
\begin{equation}\label{kfold1}
 f(e^{i\theta})= \rho(k\theta)e^{i(\theta + h(k\theta))}
\end{equation}
is well-defined and continuous on $\TT$. It has the $k$-fold symmetry~\eqref{kfold} by construction. 

Let us check that $f$ is injective. Suppose $f(e^{i\theta}) = f(e^{i\psi})$ for some $\theta,\psi \in \RR$. Then $\rho(k\theta)=\rho(k\psi)$, which by the definition of $\rho$ implies $g(k\theta)=g(k\psi)$, hence $h(k\theta)=h(k\psi)$. Comparing the arguments of $f(e^{i\theta})$  and  $f(e^{i\psi})$ we see that $\theta + h(k\theta) \equiv \psi + h(k\psi) \bmod 2\pi$.
Therefore, $\theta \equiv \psi \bmod 2\pi$ as required. 

Since $g$ is $2\pi$-periodic and even, it follows that 
\[
\begin{split}
\hat f(1) & = \frac{1}{2\pi} \int_0^{2\pi} 
\rho(k\theta)e^{i h(k\theta)}\,d\theta 
= \frac{k}{2\pi} \int_0^{2\pi/k} 
\rho(k\theta)e^{i h(k\theta)}\,d\theta \\
& = \frac{1}{2\pi} \int_0^{2\pi } 
\rho(t)e^{i h(t)}\,dt
= \frac{1}{\pi} \int_0^{\pi } 
\rho(t)e^{i h(t)}\,dt  
\end{split}
\]
But $g(t)=t$ for $t\in [0, \pi]$, which simplifies the above to 
\[
\begin{split}
\pi \hat f(1) & = \int_0^{\pi } 
\left(1+\frac{2}{\pi} t\right) e^{i (t + t^2/\pi)}\,dt 
\\ 
& =  \int_0^{2\pi }   e^{i s}\,ds = 0
\end{split}
\]
where we used $s=t + t^2/\pi $. The proof is complete.
\end{proof}

The example constructed in Theorem~\ref{kfold} is highly non-convex and is not star-shaped with respect to any point: see Figure~\ref{fig:threefold} which illustrates the case $k=3$.  

\begin{figure}[h]
    \centering
    \includegraphics[width=0.7\textwidth]{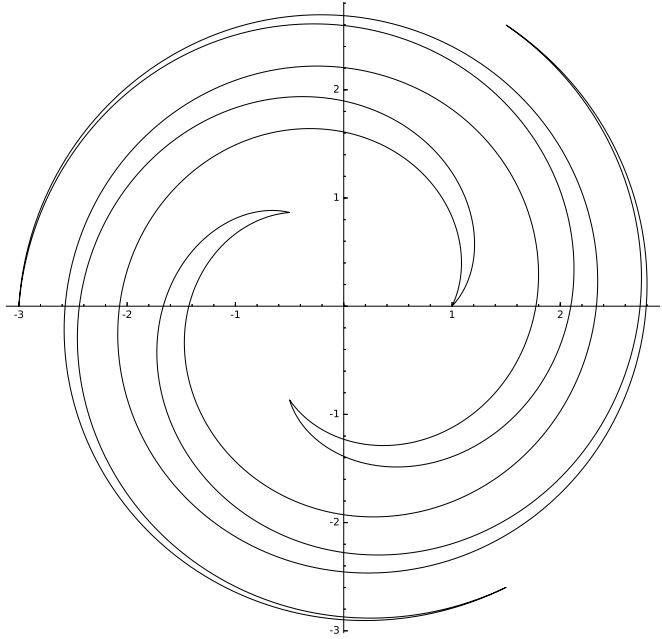}
    \caption{Three-fold symmetry with $\hat f(1) = 0 $}
    \label{fig:threefold}
\end{figure}

We obtain a nonvanishing result under the assumption of the horizontal convexity of $f(\TT)$, which is substantially weaker than convexity. Horizontal convexity naturally appears in the studies of harmonic maps~\cite{Duren}*{\S3.4}.

\begin{definition} A subset $A\subset \RR^2$ is \emph{horizontally convex} if its intersection with every horizontal line is connected (i.e., is an interval or empty set). Referring to a Jordan curve, we say that it is horizontally convex if its interior region is.   
\end{definition}

Observe that $f(\TT)$ is horizontally convex if and only if $\TT$ is the union of two arcs on each of which $\im f$ is monotone. These arcs connect the global maximum of $\im f$ on $\TT$ to its global minimum; all the extrema of $\im f$ are global.

The map in Proposition~\ref{star} shows that the horizontal convexity of $f(\TT)$ allows for $\hat f(1)=0$. We need to consider another Fourier coefficient to ensure at least one of them is nonzero.

\begin{theorem}\label{horconvex} If $f\colon\TT\to\CC$ is an embedding with horizontally convex image $f(\TT)$, then $|\hat f(-1)|+|\hat f(1)|>0$. If, in addition, $\im f$ is Lipschitz continuous, then 
\begin{equation}\label{horconvex1}
 |\hat f(-1)|+|\hat f(1)| \ge \frac{\delta}{2\pi}\left(1 - \cos \frac{\delta}{4L}\right)
\end{equation}
where $\delta = \max_\TT \im f - \min_\TT \im f$ and $L$ is the Lipschitz constant of $\im f$. 
\end{theorem}

\begin{proof} By the Borsuk-Ulam theorem (or just the intermediate value theorem), there exists $\zeta_0\in \TT$ such that $\im f(\zeta_0)=\im f(-\zeta_0)$. Replacing $f$ with $f(\alpha \zeta) + \beta$ for suitable $\alpha\in\TT$ and $\beta\in \CC$, we can arrange that $\zeta_0=1$ and $\im f(\zeta_0)=0$. Note that this replacement does not affect either side of~\eqref{horconvex1}.

By virtue of horizontal convexity, $\im f$ does not attain values of opposite sign on the upper half-circle; the same applies to the lower half-circle. Thus, $\im f(e^{i\theta}) \sin \theta$ does not attain values of opposite sign on $\TT$. Since
\begin{equation} \label{re}
\re\left(\hat f(1) - \hat f(-1)\right) =  \frac{1}{\pi} \int_0^{2\pi} \im f(e^{i\theta}) \sin \theta \,d\theta 
\end{equation}
we find that $\re\left(\hat f(1) - \hat f(-1)\right) \ne 0$ unless the integrand in ~\eqref{re} is identically zero. But the latter is impossible because the Jordan curve $f(\TT)$ cannot be contained in a line. 

It remains to prove~\eqref{horconvex1}. Pick $\zeta_1\in\TT$ such that
$|\im f|\ge \delta/2$. Let $\gamma$ be the arc of length $\delta/(2L)$ centered at $\zeta_1$. On this arc we have $|\im f|\ge \delta/4$ by the Lipschitz condition. Therefore, the absolute value of the integral in ~\eqref{re} is at least
\[
\frac{\delta}{4}\int_\gamma |\sin \theta|\,d\theta 
\ge \frac{\delta}{4}\int_{-\delta/(4L)}^{\delta/(4L)} |\sin \theta|\,d\theta = \frac{\delta}{2}\left(1 - \cos \frac{\delta}{4L}\right)
\]
proving~\eqref{horconvex1}. 
\end{proof}

Let us record an application of Theorem~\ref{horconvex} to minimal surfaces. 

\begin{corollary} Let $F\colon \DD\to\RR^3$ be a conformally parameterized minimal surface with a continuous extension to $\TT$. Let $f\colon\TT\to\RR^2$ be the composition 
of $F_{|\TT}$  with an orthogonal projection $\RR^3\to \RR^2$. If $f$ satisfies the assumptions of Theorem~\ref{horconvex}, then the Gaussian curvature $K$ of the minimal surface at $F(0)$ does not exceed  
\[
\frac{32 \pi^2 }{\delta^2\left(1 - \cos \frac{\delta}{4L}\right)^2}
\]
where $\delta$ and $L$ are as in Theorem~\ref{horconvex}.   
\end{corollary}

\begin{proof}
The computation in~\cite{Duren}*{p. 183} shows that 
\[
K \le \frac{4}{|\hat f(-1)|^2 + |\hat f(1)|^2}.
\]
On the other hand, Theorem~\ref{horconvex} implies
\[
|\hat f(-1)|^2 + |\hat f(1)|^2 \ge \frac{1}{2}(|\hat f(-1)|+|\hat f(1)|)^2 \ge 
\frac{\delta^2}{8\pi^2}\left(1 - \cos \frac{\delta}{4L}\right)^2
\]
which proves the claimed estimate.
\end{proof}

\begin{question} Is there a nonvanishing result of the form~\eqref{nonvanishing} for circle embeddings with a star-shaped image? On one hand, Hall's theorem~\cite{Hall}*{Theorem 2} gives $|\hat f(0)|+|\hat f(1)|>0$ if $f(\TT)$ is star-shaped about $0$; on the other, subtracting $\hat f(0)$ from the example in Proposition~\ref{star} shows that $|\hat f(0)|+|\hat f(1)|$ can vanish for general star-shaped embeddings.    
\end{question}

\end{document}